\documentclass[12pt]{amsart}
\usepackage[utf8]{inputenc}
\usepackage[T1]{fontenc}
\usepackage[greek,english]{babel}
\usepackage{epsfig}
\usepackage{float}
\usepackage{enumitem}
\usepackage{graphicx}
\usepackage{float}
\usepackage{tcolorbox}
\usepackage{amsfonts,amsmath,amssymb,amsthm,hyperref,setspace,verbatim,times,longtable}
\usepackage{graphicx, complexity}
\usepackage{MnSymbol}
\usepackage{upquote}
\usepackage[a4paper,margin=1.9cm,top=2.5cm,bottom=2.5cm,centering,vcentering]{geometry}

\newcommand{\abs}[1]{\lvert#1\rvert}

\DeclareMathOperator{\ad}{AD}
\DeclareMathOperator{\mc}{MC}
\DeclareMathOperator{\omc}{\overline{MC}}

\DeclareMathOperator{\supp}{Supp}
\usepackage{graphicx}

\makeatletter
\newcommand*\bigcdot{\mathpalette\bigcdot@{1}}
\newcommand*\bigcdot@[2]{\mathbin{\vcenter{\hbox{\scalebox{#2}{$\m@th#1\bullet$}}}}}
\makeatother

\usepackage{color}
\definecolor{ao(english)}{rgb}{0.0, 0.0, 1.0}
\hypersetup{colorlinks=true, linkcolor=ao(english),citecolor=ao(english)}

\usepackage[normalem]{ulem}

\usepackage{color}

\newtheorem{theorem}{Theorem}

\newtheorem{definition}{Definition}
\newtheorem{lemma}{Lemma}
\newtheorem{remark}[theorem]{Remark}

\title[Symmetric Domino Tilings of Aztec Diamonds]{Symmetric Domino Tilings of Aztec Diamonds}

\author[P. Paul]{Pravakar Paul}
\address{Mathematical and Physical Sciences division, School of Arts and Sciences, Ahmedabad University, Ahmedabad 380009, Gujarat, India}
\email{pravakar.paul@ahduni.edu.in, manjil@saikia.in}

\author[M. P. Saikia]{Manjil P. Saikia}
\keywords{Perfect matchings, domino tilings, Aztec diamonds.}

\subjclass[2010]{Primary 05C30; Secondary 05A15, 05C70, 52C20, 68R10}

\thanks{The work of the second author is supported by an Ahmedabad University Start-Up Grant (Reference No. URBSASI24A5).}

\begin{document}

\begin{abstract}
In this paper, we give inductive sum formulas to calculate the number of diagonally symmetric, and diagonally \& anti-diagonally symmetric domino tilings of Aztec Diamonds. As a byproduct, we also find such a formula for the unrestricted case as well. Our proofs rely on a new technique for counting the number of perfect matchings of graphs, proposed by the authors recently.
\end{abstract}

\maketitle

\section{Introduction}

In 1992, Elkies, Kuperberg, Larsen, and Propp \cite{AD1, AD2} introduced a new class of objects which they called Aztec diamonds. The Aztec diamond of order $n$ (denoted by $\ad(n)$) is the union 
of all unit squares inside the contour $\abs{x}+\abs{y}=n+1$ (see the top left of Figure \ref{fig:ad} for an Aztec diamond of order $4$; at this moment we ignore the other figures). They considered the problem of counting the number of domino tilings of an Aztec diamond of order $n$ (see the top right of Figure \ref{fig:ad} for an example of such a tiling for $n=4$) and proved that this number is equal to $2^{n(n+1)/2}$. A domino tiling of $\ad(n)$ is a covering of $\ad(n)$ using dominoes without gaps or overlaps. They gave four different proofs of this result.
\begin{theorem}[Aztec Diamond Theorem, \cite{AD1, AD2}]\label{adm}
The number of domino tilings of the Aztec diamond of order $n$ is $2^{n(n+1)/2}$.
\end{theorem}

\begin{figure}[!htb]
\centering
\includegraphics[scale=.79]{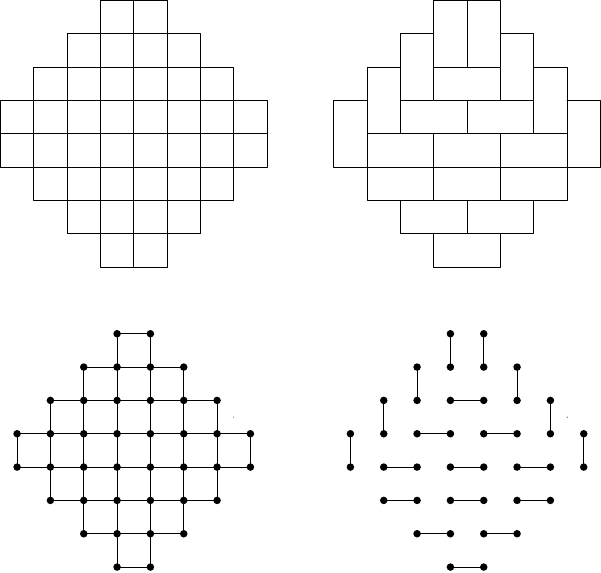}
\caption{Top row: the Aztec Diamond $\ad(n4)$ (left) and a tiling of $\ad(n4)$ (right); bottom row: the planar dual graph of $\ad(n4)$ (left) and the perfect matching of the planar dual graph of $\ad(n4)$ corresponding to the tiling directly above it (right).}
\label{fig:ad}
\end{figure}

The proofs given by Elkies et. al. \cite{AD1, AD2} gave insights into the connection of Aztec Diamonds with other objects, namely alternating sign matrices (ASMs), monotone triangles, Grassmann algebras, square ice, etc. Several generalizations of Aztec Diamonds have also appeared in the literature, for instance Aztec rectangles \cite{Saikia}, quartered Aztec Dimonds \cite{l-1, l-5}, hybrid regions on the square and triangular lattices \cite{l-1, l-4}, Aztec triangles \cite{t-1, t-2, t-3}, etc. In other direction, asymptotics of the domino tiling model on a class of domains related to Aztec Diamonds have also been studied \cite{bk}. The Arctic circle phenomenon for Aztec Diamonds and some of its generalizations have also well-studied \cite{shor, cohn, zfr}. All of these makes the Aztec Diamonds one of the most studied classes of combinatorial objects, with connections to different areas of mathematics, including but not limited to probability theory (see, for instance \cite{ch-2, ch-1}), statistical physics (see, for instance \cite{statphy, statphy-2}), other areas of combinatorics (via its connections to ASMs), etc.

One combinatorial aspect of Aztec Diamonds that deserve proper study are the symmetric domino tilings. Indeed, the symmetry classes of domino tilings of $\ad(n)$ under the action of the dihedral group of order $8$ have been discussed in the literature by Ciucu \cite{CiucuSymmetry} and Yang \cite{Yang}. In total, there are five symmetry classes; the enumerations of three of which have been solved completely, while the other two remain open with no known conjectured formulas. In this paper, we look at these two remaining symmetry classes. 

 Let $r$ and $t$ be the generators corresponding to a rotation by $90^\circ$ and a reflection across the vertical diagonal of $\ad(n)$, respectively. Then, the dihedral group of order $8$ is
\[
D_8=\langle r, t| r^4=t^2=(tr)^2=\text{id}\rangle.
\]
We list the symmetry classes in Table \ref{table1}. The first few values for the number of domino tilings for each case are also given in the last column of Table \ref{table1}.
\begin{table}[htb]\label{table1}
\begin{tabular}{l|l|l}
Subgroup of $D_8$ & Symmetry Class & Sequence of values  \\ \hline
\{\text{id}\} & Unrestricted Aztec Diamonds & $2, 8, 64, 1024, 32768,2097152, \ldots$  \\
$\langle r \rangle $ & Quarter-turn symmetric & $0,0,2,4,0,0,80, 320, \ldots$   \\
$\langle r^2 \rangle $ & Half-turn symmetric  & $ 2, 4, 12, 48, 288, 2304, \ldots$ \\
$\langle t \rangle$ & Diagonally symmetric & $2, 6, 24, 132, 1048, 11960, 190912,\ldots$ \\
$\langle r^2, t \rangle$ & Diagonally \& anti-diagonally symmetric & $2, 4, 10, 28, 96, 384, 1848,  \ldots$\\\hline
\end{tabular}
\caption{Symmetry Classes of Aztec Diamonds.}
\end{table}
As already mentioned, a formula for the unrestricted case is given by Theorem \ref{adm}. A formula for the quarter-turn symmetric case and the half-turn symmetric case can be found in the work of Ciucu \cite{CiucuSymmetry} and Yang \cite{Yang}. No formula is known or conjectured for the diagonally symmetric and the diagonally \& anti-diagonally symmetric cases. In some recent work, Lee \cite{Lee1, Lee2} has looked into two more classes of Aztec Diamonds not covered in the previous five classes; however, we do not discuss them here.

\begin{table}[htb]\label{table2}
\begin{tabular}{l|l|l}
$n$ & $\langle t \rangle$-symmetric domino tilings & $\langle r^2, t \rangle$-symmetric domino tilings \\ \hline
$1$ & $2^1$ & $2^1$  \\
$2$ & $2^1\cdot 3^1$ & $2^2$   \\
$3$ & $2^3\cdot 3^1$  & $2^1\cdot 5^1$ \\
$4$ & $2^2\cdot 3\cdot 11$ & $2^2\cdot 7$ \\
$5$ & $2^{10}$ & $2^5\cdot 3^1$\\
$6$ & $2^3\cdot 5^1\cdot 13^1\cdot 23^1$ & $2^7\cdot 3^1$\\
$7$ & $2^6\cdot 19^1\cdot 157^1$ & $2^3 \cdot 3^1\cdot 7^1\cdot 11^1$ \\ \hline
\end{tabular}
\caption{Prime Factorizations of small values of the number of symmetric domino tilings.}
\end{table}

Ideally, combinatorialists prefer a `product' (or, in some cases, a `sum') formula of the type in Theorem \ref{adm}, but due to the presence of large prime factors in the prime factorization of the values in Table \ref{table2} for the diagonally symmetric and the diagonally \& anti-diagonally symmetric cases, it seems unlikely that a product formula exists for these two cases. In this paper, we give an inductive sum formula for these two remaining symmetry classes (Theorems \ref{thm2} and \ref{thm5}). We also give a similar formula for the unrestricted Aztec Diamond case (Theorem \ref{thm3}). The novelty of our approach here is that, unlike trying via a brute-force way to calculate the number of domino tilings of the regions, our formula builds up from the bottom values and gives a clean and inductive way to calculate the number of domino tilings of the desired regions. The main ingredient that we use is a new approach to calculating the number of perfect matchings of a planar graph discussed in some recent work by us \cite{PaulSaikia}. We defer discussion on this towards the later part of the paper. In our earlier work \cite{PaulSaikia} we gave an independent proof of Theorem \ref{adm} using our new techniques. Our proof also brings our the connection between Aztec Diamonds and ASMs in a much more natural way. This approach was also used recently by the first author \cite{pp} to shed some light on a tiling problem on the triangular lattice.

This paper is organized as follows: in Section \ref{sec:two} we present the inductive maps that we use to calculate the number of unrestricted and diagonally-symmetric domino tilings of $\ad(n)$, in Section \ref{sec:three} we present the inductive maps to calculate the number of diagonally \& anti-diagonally symmetric domino tilings of $\ad(n)$; the proofs of the results mentioned in Sections \ref{sec:two} and \ref{sec:three} are given in Section \ref{sec:5}.

\section{Diagonally Symmetric Domino Tilings of Aztec Diamonds}\label{sec:two}

In this section, we present a way to inductively compute the number of diagonally symmetric domino tilings of Aztec Diamonds in Theorem \ref{thm2}. We also give a similar result for the unrestricted case in Theorem \ref{thm3}. Since the construction of some of the maps we use in our results is intricate, we push the proofs of the theorems to Section \ref{sec:5}.

We define a collection of functions $\{C_{n} \}_{n \in \mathbb{N}}$ as follows
\[ C_{n}: \{ 0,1 \}^{2n} \to \mathbb{Z}^{+}, \] where $\{0,1 \}^{2n}$ denotes the collection of binary sequences of length $2n$ and $\mathbb{Z}^{+}$ denotes the set of non-negative integers.  We associate the following important quantity to each $C_n$. 
\begin{definition}[$l^k$- norm of $C_n$]
    The $l^{k}$-norm of $C_n$ is defined by the following expression 
\[ ||C_{n}||_{l^{k}} :=  \sum_{\epsilon_{i} \in \{0,1 \}} \left(C_{n}\left(\epsilon_1, \ldots, \epsilon_{2n} \right)\right)^k. \]
\end{definition}
\noindent Before defining the functions $C_{n}$, we list a property which we prove in Section \ref{sec:5}. 
\begin{lemma}\label{lem1}
     The support of $C_{n}$ is bounded above by the following quantity \[ \supp(C_{n}) \leq \left\lfloor  \frac{2^{2n}}{3} \right\rfloor +1, \]
    where $\supp(C_n)$ denotes the number of elements in the domain with a non-zero image.
\end{lemma}    
\begin{theorem}\label{thm2}
 The $l^{1}$ norm of $C_{n}$, $||C_{n}||_{l^{1}}$ enumerates the number of diagonally symmetric, that is $\left< t \right>$-invariant domino tilings of $\ad(n)$. 
\end{theorem}

\begin{theorem}\label{thm3}
 The $l^{2}$ norm of $C_{n}$, $||C_{n}||_{l^{2}}$ enumerates the number of domino tiling of $\ad(n)$, which is given by the formula
    \[ ||C_{n}||_{l^{2}}= 2^{\frac{(n+1)\cdot n}{2}}.\]
\end{theorem}

\noindent We illustrate Theorems \ref{thm2} and \ref{thm3} at the end of this section for $n=2$.

Next, we move to the definition of $C_{n}$. We provide an inductive definition of $ \{ C_{n} \}_{n \in \mathbb{N}}$. The base case $n=1$ is given by the following formula 
\begin{equation}
    C_{1}(0,0) = 1, \quad C_{1}(0,1)= 0, \quad   C_{1}(1,0)= 0, \quad C_{1}(1,1) = 1.
\end{equation}
Inductively we shall define $C_{n}$ using $C_{n-1}$. To do so, we need a few more constructions.

\subsection*{Construction on free abelian group.}
Any set map \( f: S \longrightarrow T \) canonically extends to a map \(f: \mathbb{Z} \left<  S \right> \longrightarrow \mathbb{Z} \left< T \right>  \), where $\mathbb{Z} \left< S \right>$ and $\mathbb{Z} \left< T \right>$ denote the free abelian group generated by $S$ and $T$ respectively (see for instance \cite{DummitFoote}). By an abuse of notation, we shall denote the map on free abelian groups by the same letter $f$. 

Similarly, any set map $g: S \longrightarrow \mathbb{Z}^{+}$ canonically extends to a unique homomorphism of abelian groups 
\( g: \mathbb{Z} \left< S \right> \longrightarrow \mathbb{Z}\). Again, by an abuse of notation, we shall use the letter $g$ to denote the unique homomorphism. 

\subsection*{The reverse map $R_{n}$.} We define a map called the reverse map $R_n$ as follows 
\[R_{n}: \{ 0,1 \}^{n} \to \{0,1 \}^{n} \]
\begin{align*}
    R_{n} \left( \epsilon_1, \ldots , \epsilon_n \right) &= (\bar{\epsilon_1}, \ldots, \bar{\epsilon_n}),
\end{align*}
where $\{ \epsilon_i,  \bar{\epsilon_{i}} \} = \{0, 1 \}$, for all $1\leq i\leq n$. By the previous paragraph the symbol $R_{n}$ shall also denote the induced map on the respective free abelian groups. Note that $R_{n}$ is a fixed point free involution. 

\subsection*{The expansion map $E_{n}$.} For all $ n \geq 2$ we define a collection of maps called the expansion maps $\{E_{n} \}_{n \geq 2}$. The expansion map $E_{n}$ is a map of the following form
  \begin{align*}
      E_{n}: \mathbb{Z} \left< \{0,1 \}^{n} \right> &\longrightarrow \mathbb{Z} \left< \{0,1 \}^{n} \right> \\
      E_{n}&=    E^{n-1}_{n} \circ \cdots \circ E^{1}_{n},  
  \end{align*}
where we are composing the the maps $E^{i}_{n}$, which we define now. For $ 1 \leq i \leq n-1 $, the maps $E^{i}_{n}$ are called the basic expansions, which are given by the following formulas 
  \begin{align*}
      E^{i}_{n}: \mathbb{Z} \left< \{0,1 \}^{n} \right> &\longrightarrow \mathbb{Z} \left< \{0,1 \}^{n} \right> \\
      E^{i}_{n} \Bigl((\epsilon_1, \ldots, \epsilon_{i-1}, 0,0, \ldots ,\epsilon_{n}) \Bigl)&= (\epsilon_1, \ldots, \epsilon_{i-1}, 0,0, \ldots ,\epsilon_{n}) + (\epsilon_1, \ldots, \epsilon_{i-1}, 1,1, \ldots ,\epsilon_{n}), \\ 
      E^{i}_{n} \Bigl((\epsilon_1, \ldots, \epsilon_{i-1}, 0,1, \ldots ,\epsilon_{n}) \Bigl)&= (\epsilon_1, \ldots, \epsilon_{i-1}, 0,1, \ldots ,\epsilon_{n}), \\
      E^{i}_{n} \Bigl((\epsilon_1, \ldots, \epsilon_{i-1}, 1,0, \ldots ,\epsilon_{n}) \Bigl)&= (\epsilon_1, \ldots, \epsilon_{i-1}, 1,0, \ldots ,\epsilon_{n}), \\
      E^{i}_{n} \Bigl((\epsilon_1, \ldots, \epsilon_{i-1}, 1,1, \ldots ,\epsilon_{n}) \Bigl)&= (\epsilon_1, \ldots, \epsilon_{i-1}, 1,1, \ldots ,\epsilon_{n}).
  \end{align*}
The map $E_3$ is explicitly calculated below. In order to calculate the expansion map $E_{3}$, we need to first calculate the basic expansion maps $E^{1}_{3}$ and $E^{2}_{3}$, both of which are done now.

First, we calculate $E_3^1$:
\begin{align*}
      E^{1}_{3}: \mathbb{Z} \left< \{0,1 \}^{3} \right> &\longrightarrow \mathbb{Z} \left< \{0,1 \}^{3} \right>, \\
      E^{1}_{3} \bigl( (0,0,0) \bigl) &= (0, 0, 0) + (1, 1, 0), \\
      E^{1}_{3} \bigl( (0,0,1) \bigl) &= (0, 0, 1) + (1, 1, 1), \\
      E^{1}_{3} \bigl( (0,1,0) \bigl) &= (0, 1, 0), \\ 
      E^{1}_{3} \bigl( (0,1,1) \bigl) &= (0, 1, 1), \\ 
      E^{1}_{3} \bigl( (1,0,0) \bigl) &= (1, 0 , 0), \\
      E^{1}_{3} \bigl( (1,0,1) \bigl) &= (1, 0, 1), \\
      E^{1}_{3} \bigl( (1,1,0) \bigl) &= (1, 1, 0), \\
      E^{1}_{3} \bigl( (1,1,1) \bigl) &= (1, 1, 1).
  \end{align*}
  Similarly, $E_3^2$ is as follows:
  \begin{align*}
      E^{2}_{3}: \mathbb{Z} \left< \{0,1 \}^{3} \right> &\longrightarrow \mathbb{Z} \left< \{0,1 \}^{3} \right>, \\
      E^{2}_{3} \bigl( (0,0,0) \bigl) &= (0, 0, 0) + (0, 1, 1), \\
      E^{2}_{3} \bigl( (0,0,1) \bigl) &= (0, 0, 1),  \\
      E^{2}_{3} \bigl( (0,1,0) \bigl) &= (0, 1, 0), \\ 
      E^{2}_{3} \bigl( (0,1,1) \bigl) &= (0, 1, 1) \\ 
      E^{2}_{3} \bigl( (1,0,0) \bigl) &= (1, 0 , 0) + (1, 1, 1),\\
      E^{2}_{3} \bigl( (1,0,1) \bigl) &= (1, 0, 1), \\
      E^{2}_{3} \bigl( (1,1,0) \bigl) &= (1, 1, 0), \\
      E^{2}_{3} \bigl( (1,1,1) \bigl) &= (1, 1, 1).
  \end{align*}
Finally the expansion map $E_{3}$  is given by the following formula: 
\begin{align*}
    E_{3}: \mathbb{Z} \left< \{0,1 \}^{3} \right> &\longrightarrow \mathbb{Z} \left< \{0,1 \}^{3} \right>, \\
      E_{3} \bigl( (0,0,0) \bigl) &= (0, 0, 0) + (1, 1, 0)+ (0, 1, 1), \\
      E_{3} \bigl( (0,0,1) \bigl) &= (0, 0, 1) +(1, 1, 1 ), \\
      E_{3} \bigl( (0,1,0) \bigl) &= (0, 1, 0), \\ 
      E_{3} \bigl( (0,1,1) \bigl) &= (0, 1, 1) ,\\ 
      E_{3} \bigl( (1,0,0) \bigl) &= (1, 0 , 0) + (1, 1, 1),\\
      E_{3} \bigl( (1,0,1) \bigl) &= (1, 0, 1), \\
      E_{3} \bigl( (1,1,0) \bigl) &= (1, 1, 0), \\
      E_{3} \bigl( (1,1,1) \bigl) &= (1, 1, 1).
\end{align*} 
\subsection*{The contraction map $K_{n}$.} For $ n \geq 4 $, we define a collection of maps called the contraction maps $\{ K_{n} \}_{n \geq 4}$. The contraction map $K_{n}$ is a map of the following form
 \[ K_{n}: \mathbb{Z} \left< \{0,1 \}^{n} \right>  \longrightarrow \mathbb{Z} \left< \{0,1 \}^{n-2} \right> \]

%%%%%%%%%%NO NEED TO EXPLAIN THE PATTERN%%%%%%%%%%%%%%%%%%%%%%%%%%%%%
%%%%%%%%%%%%%%%%%%%%%%%%%%%%%%%%%%%%%%%%%%%%%%%%%%%%%%%%%%%%%%%%%%%%%%%
%%%%%%%%%%%%%%%%%%%%%%%%%%%%%%%%%%%%%%%%%%%%%%%%%%%%%%%%%%%%%%%%%%

% The map $K_{n}$ is given by the following rules:
% \begin{itemize}
%     \item when the sequence starts with $(0,0, \ldots)$ replace the initial two $0's$ by a single $1$,
%    \item when the sequence starts with $(0,1, \ldots)$ then it goes to the zero element $0$ in the abelian group,
%     \item when the sequence starts with $(1,0, \ldots)$ we replace the initial $(1,0)$ by a single $0$, 
 
%\item when the sequence starts with $(1,1,\ldots)$ replace the initial two $1's$ by a single $1$, 
%    \item when the sequence ends with $(\ldots, 0,0)$ replace the final two $0's$ by a single $1$,
%    \item when the sequence ends with $(\ldots, 0,1)$ we replace the final $(0,1)$ by a single $0$, 
%    \item when the sequence starts with $( \ldots, 1, 0)$ then it goes to the zero element $0$ in the abelian group, and
%     \item when the sequence ends with $(\ldots, 1,1)$ we replace the final two $1's$ by a single $1$. 
%\end{itemize}

More explicitly the map $K_{n}$ is given below:
\begin{align*}
    K_{n} \bigl( (0,0, \epsilon_3, \ldots, \epsilon_{n-2}, 0,0) \bigl) &= (1, \epsilon_3, \ldots, \epsilon_{n-2}, 1), \\ 
    K_{n} \bigl( (0,0, \epsilon_3, \ldots, \epsilon_{n-2}, 0,1) \bigl) &= (1, \epsilon_3, \ldots, \epsilon_{n-2}, 0), \\
    K_{n} \bigl( (0,0, \epsilon_3, \ldots, \epsilon_{n-2}, 1,0) \bigl) &= 0, \\
    K_{n} \bigl( (0,0, \epsilon_3, \ldots, \epsilon_{n-2}, 1,1) \bigl) &= (1, \epsilon_3, \ldots, \epsilon_{n-2}, 1), \\
    K_{n} \bigl( (0,1, \epsilon_3, \ldots, \epsilon_{n-2}, 0,0) \bigl) &= 0,  \\
    K_{n} \bigl( (0,1, \epsilon_3, \ldots, \epsilon_{n-2}, 0,1) \bigl) &= 0, \\ 
    K_{n} \bigl( (0,1, \epsilon_3, \ldots, \epsilon_{n-2}, 1,0) \bigl) &= 0, \\
    K_{n} \bigl( (0,1, \epsilon_3, \ldots, \epsilon_{n-2}, 1,1) \bigl) &= 0, \\
    K_{n} \bigl( (1,0, \epsilon_3, \ldots, \epsilon_{n-2}, 0,0) \bigl) &= (0, \epsilon_3, \ldots, \epsilon_{n-2}, 1), \\
    K_{n} \bigl( (1,0, \epsilon_3, \ldots, \epsilon_{n-2}, 0,1) \bigl) &= (0, \epsilon_3, \ldots, \epsilon_{n-2}, 0), \\
    K_{n} \bigl( (1,0, \epsilon_3, \ldots, \epsilon_{n-2}, 1,0) \bigl) &= 0, \\
    K_{n} \bigl( (1,0, \epsilon_3, \ldots, \epsilon_{n-2}, 1,1) \bigl) &= (0, \epsilon_3, \ldots, \epsilon_{n-2}, 1), \\
    K_{n} \bigl( (1,1, \epsilon_3, \ldots, \epsilon_{n-2}, 0,0) \bigl) &= (1, \epsilon_3, \ldots, \epsilon_{n-2}, 1), \\
    K_{n} \bigl( (1,1, \epsilon_3, \ldots, \epsilon_{n-2}, 0,1) \bigl) &= (1, \epsilon_3, \ldots, \epsilon_{n-2}, 0), \\
    K_{n} \bigl( (1,1, \epsilon_3, \ldots, \epsilon_{n-2}, 1,0) \bigl) &= 0, \\
    K_{n} \bigl( (1,1, \epsilon_3, \ldots, \epsilon_{n-2}, 1,1) \bigl) &= (1, \epsilon_3, \ldots, \epsilon_{n-2}, 1).
\end{align*}

Although this definition of $K_n$ may appear ad-hoc, its origin will become clear as we proceed in our discussion.

\subsection*{The inductive definition of $C_n$.} 
Finally, with all the above background we can write down the formula for $C_n$ in terms of $C_{n-1}$ as follows
\[ C_{n}=   C_{n-1} \circ   R_{2n-2} \circ E_{2n-2} \circ K_{2n}  \]
where we treat the maps $C_{n}$ as the homomorphism of abelian groups as discussed above.

We illustrate the calculation for Theorems \ref{thm2} and \ref{thm3} when $n=2$ in Table \ref{table3}.
\begin{table}[htb]\label{table3}
\begin{tabular}{l|l|l|l|l}
$\epsilon=(\epsilon_1, \epsilon_2,\epsilon_3,\epsilon_4)$ & $K_4(\epsilon)$ & $E_2( K_4(\epsilon))$ & $R_2( E_2( K_4(\epsilon)))$ & $C_2(\epsilon)=C_1( R_2(E_2(K_4(\epsilon))))$\\ \hline
$(0,0,0,0)$ & $(1,1)$ & $(1,1)$ & $(1,1)$ & $1$ \\
$(0,0,0,1)$ & $(1,0)$ & $(1,0)$ & $(0,1)$ & $0$ \\
$(0,0,1,0)$ & $0$ & $0$ & $0$ & $0$ \\
$(0,0,1,1)$ & $(1,1)$ & $(1,1)$ & $(1,1)$ & $1$ \\
$(0,1,0,0)$ & $0$ & $0$ & $0$ & $0$ \\
$(0,1,0,1)$ & $0$ & $0$ & $0$ & $0$ \\
$(0,1,1,0)$ & $0$ & $0$ & $0$ & $0$ \\
$(0,1,1,1)$ & $0$ & $0$ & $0$ & $0$ \\
$(1,0,0,0)$ & $(0,1)$ & $(0,1)$ & $(1,0)$ & $0$ \\
$(1,0,0,1)$ & $(0,0)$ & $(0,0)+(1,1)$ & $(0,0)+(1,1)$ & $2$ \\
$(1,0,1,0)$ & $0$ & $0$ & $0$ & $0$ \\
$(1,0,1,1)$ & $(0,1)$ & $(0,1)$ & $(1,0)$ & $0$ \\
$(1,1,0,0)$ & $(1,1)$ & $(1,1)$ & $(1,1)$ & $1$ \\
$(1,1,0,1)$ & $(1,0)$ & $(1,0)$ & $(0,1)$ & $0$ \\
$(1,1,1,0)$ & $0$ & $0$ & $0$ & $0$ \\
$(1,1,1,1)$ & $(1,1)$ & $(1,1)$ & $(1,1)$ & $1$ \\
\hline
\end{tabular}
\caption{Illustration of the map $C_2$.}
\end{table}
\noindent The $l^1$ norm of $C_2$ is then just the sum of the entries in the last column of Table \ref{table3}, which equals $6$ and this matches with the data in Table \ref{table1}, giving us the number of $\langle t\rangle$-invariant domino tilings of $\ad(2)$. The $l^2$ norm of $C_2$ is then just the sum of the squares of the entries in the last column of Table \ref{table3}, which equals $8$ and this matches with the data in Table \ref{table1}, giving us the number of unrestricted domino tilings of $\ad(2)$.

\section{Diagonally and Antidiagonally Symmetric Domino Tilings of Aztec Diamonds}\label{sec:three}

In this section, we present a way to inductively compute the number of diagonally \& anti-diagonally symmetric domino tilings of Aztec Diamonds in Theorem \ref{thm5}. Since the construction of some of the maps we use in our result is intricate, we push the proof of the theorem to Section \ref{sec:5}.

We define a collection of functions $\{ C^\prime_{n} \}_{n \in \mathbb{N} }$ as follows: 
\[C^\prime_{n}: \{0,1 \}^{2n+1} \to \mathbb{Z}^{+}. \]
We define $C^\prime_1$ and $C^\prime_{2}$ separately. Inductively, 
$C^\prime_{n+2}$ shall be defined using $C^\prime_{n}$. Elements in $\{0,1 \}^{2n+1}$ can be grouped into two parts: 
the first part consists of elements where the middle element $\epsilon_{n+1}$ is $0$ and the second part consists of elements where the middle element $\epsilon_{n+1}$ is $1$. Given such a function we associate the following important quantity with it.

\begin{definition}[$l^{1/2}$-norm of $C^\prime_{n}$]
    The $l^{1/2}$-norm of $C^\prime_{n}$ is defined as
\begin{align*}
||C^\prime_{n}||_{l^{1/2}} :=&   \sum_{\epsilon_i \in \{0,1\} \, \, \text{\&} \, \, \epsilon_{n+1}=0}  C^\prime_{n} \bigl( ( \epsilon_1, \ldots, \epsilon_n, 0, \epsilon_{n+2}, \ldots, \epsilon_{2n
+1} ) \bigl) \\ & \quad + \sum_{\epsilon_i \in \{0,1\} \, \, \text{\&} \, \, \epsilon_{n+1}=1} 2 \cdot C^\prime_{n} \bigl( ( \epsilon_1, \ldots, \epsilon_n, 1, \epsilon_{n+2}, \ldots, \epsilon_{2n
+1} ) \bigl) .
\end{align*}
\end{definition}

\begin{theorem}\label{thm5}
 The $l^{1/2}$ norm of $C^\prime$, $||C^\prime_{n}||_{l^{1/2}}$ enumerates the number of diagonally and anti-diagonally symmetric, that is $\left< r^2, t \right>$ invariant domino tilings of $\ad(n+1)$.
\end{theorem}
%\noindent We illustrate Theorem \ref{thm5} at the end of this section for $n=2$.

We give a proof of Theorem \ref{thm5} in the next section. In this section, we focus on the construction of the map $C^\prime_{n}$. To begin with, we write down the formula for $C^\prime_{1}$ and $C^\prime_{2}$. We first start with the definition of $C^\prime_{1}$.
\begin{equation}
     C^\prime_{1} \bigl( (1,1,1) \bigl) = 1, \quad C^\prime_{1} \bigl( (1,0,0) \bigl) = 1, \quad C^\prime_{1} \bigl( (0,0,1) \bigl) = 1, \quad  C^\prime_{1} \bigl( (\epsilon_1,\epsilon_2,\epsilon_3) \bigl) = 0 \, \, \text{otherwise}.
\end{equation}
Note that $||C^\prime_{1}||_{l^{1/2}} = 4$, which agrees with the value in Table \ref{table2}. Next, we shall define $C^\prime_{2}$.
\begin{multline}
     C^\prime_{2} \bigl( (0,0,1,0,1) \bigl)=1, \quad C^\prime_{2} \bigl( (1,0,1,0,0) \bigl)=1, \quad C^\prime_{2} \bigl( (1,0,1,1,1) \bigl) = 1, \quad C^\prime_{2} \bigl( (1,1,1,0,1) \bigl)=1,\\
     C^\prime_{2} \bigl( (1,0,0,0,1) \bigl)=2, \quad C^\prime_{2} \bigl( (\epsilon_1,\epsilon_2,\epsilon_3,\epsilon_4,\epsilon_5) \bigl) = 0 \, \, \text{otherwise}.
\end{multline}
Note that, $||C^\prime_{2}||_{l^{1/2}} = 10$, which agrees with the value in Table \ref{table2}. To define $C^\prime_{n+2}$ from $C^\prime_{n}$ we shall require a few constructions as before. 

\subsection*{Further constructions on free abelian groups.}

Notice that there is a canonical bijection 
\[ \{ 0,1 \}^{n_1+n_2}  \xrightarrow{\sim}  \{ 0,1 \}^{n_1} \times \{ 0, 1 \}^{n_2}. \]
The map is given by \[  \left( \epsilon_1, \ldots, \epsilon_{n_1}, \epsilon_{n_1+1}, \ldots, \epsilon_{n_{1}+n_2 }) \mapsto \bigl( (\epsilon_1, \ldots, \epsilon_{n_1}), (\epsilon_{n_1+1}, \ldots, \epsilon_{n_1+n_2})   \right) \]
This canonical bijection extends to a canonical isomorphism of free abelian groups as follows: \[  \mathbb{Z} \left< \{ 0,1 \}^{n_1+n_2} \right> \xrightarrow{\sim} \mathbb{Z} \left< \{ 0,1 \}^{n_1} \right> \otimes \mathbb{Z} \left< \{ 0,1 \}^{n_2} \right>  \]

\noindent In general, given two maps of abelian groups
\begin{align*}
    f: A_1 \longrightarrow  \mathbb{Z} \left< \{ 0,1 \}^{n_1} \right> &\quad \text{and}\quad 
    g: A_2 \longrightarrow  \mathbb{Z} \left< \{ 0,1 \}^{n_2} \right>.
\end{align*}
It naturally defines a map  
\[ f \otimes g: A_1 \otimes A_2 \longrightarrow \mathbb{Z} \left< \{ 0,1 \}^{n_1} \right> \otimes \mathbb{Z} \left< \{ 0,1 \}^{n_2} \right> \xrightarrow{\sim} \mathbb{Z} \left< \{ 0,1 \}^{n_1+n_2} \right> \]
Similar natural maps exist if we reverse the arrows of $f$, $g$, and $f \otimes g$ in the previous construction. 

\subsection*{The middle contraction.} We define a collection of maps called the middle contraction maps $\{ mK_{n} \}_{n \geq 1}$. The middle contraction map $mK_{n}$ is a map of the following form: 
\[mK_{n}: \mathbb{Z} \left< \{0,1 \}^{2n+1} \right> \longrightarrow \mathbb{Z} \left< \{0,1 \}^{2n} \right>  \]
\begin{align*}
    mK_{n}\bigl( (\epsilon_1, \ldots, \epsilon_{n-1}, 1,1,1, \epsilon_{n+3}, \ldots, \epsilon_{2n+1}) \bigl) &= (\epsilon_1, \ldots, \epsilon_{n-1}, 1,1,\epsilon_{n+3}, \ldots, \epsilon_{2n+1})), \\ 
     mK_{n}\bigl( (\epsilon_1, \ldots, \epsilon_{n-1}, 1,0,0, \epsilon_{n+3}, \ldots, \epsilon_{2n+1}) \bigl) &= (\epsilon_1, \ldots, \epsilon_{n-1}, 1,1,\epsilon_{n+3}, \ldots, \epsilon_{2n+1})), \\
     mK_{n}\bigl( (\epsilon_1, \ldots, \epsilon_{n-1}, 0,0,1, \epsilon_{n+3}, \ldots, \epsilon_{2n+1}) \bigl) &= (\epsilon_1, \ldots, \epsilon_{n-1}, 1,1,\epsilon_{n+3}, \ldots, \epsilon_{2n+1})), \\ 
     mK_{n}\bigl( (\epsilon_1, \ldots, \epsilon_{n-1}, 0,1,1, \epsilon_{n+3}, \ldots, \epsilon_{2n+1}) \bigl) &= (\epsilon_1, \ldots, \epsilon_{n-1}, 0,1,\epsilon_{n+3}, \ldots, \epsilon_{2n+1})), \\
     mK_{n}\bigl( (\epsilon_1, \ldots, \epsilon_{n-1}, 1,1,0, \epsilon_{n+3}, \ldots, \epsilon_{2n+1}) \bigl) &= (\epsilon_1, \ldots, \epsilon_{n-1}, 1,0,\epsilon_{n+3}, \ldots, \epsilon_{2n+1})), \\
     mK_{n}\bigl( (\epsilon_1, \ldots, \epsilon_{n-1}, 0,0,0, \epsilon_{n+3}, \ldots, \epsilon_{2n+1}) \bigl) &= (\epsilon_1, \ldots, \epsilon_{n-1}, 0,1,\epsilon_{n+3}, \ldots, \epsilon_{2n+1})) \\ 
     &\quad + (\epsilon_1, \ldots, \epsilon_{n-1}, 1,0,\epsilon_{n+3}, \ldots, \epsilon_{2n+1})), \\ 
     mK_{n}\bigl( (\epsilon_1, \ldots, \epsilon_{n-1}, 0,1,0, \epsilon_{n+3}, \ldots, \epsilon_{2n+1}) \bigl) &= (\epsilon_1, \ldots, \epsilon_{n-1}, 0,0,\epsilon_{n+3}, \ldots, \epsilon_{2n+1})), \\
     mK_{n}\bigl( (\epsilon_1, \ldots, \epsilon_{n-1}, 1,0,1, \epsilon_{n+3}, \ldots, \epsilon_{2n+1}) \bigl) &= 0. 
\end{align*}
Observe that, we have the general formula 
\[ mK_{n}= Id_{n-1} \otimes  mK_{1} \otimes Id_{n-1}. \]
We need one more map to write down the formula for $C^\prime_{n+2}$.

\subsection*{The central contraction.} We define a collection of maps called the central contraction maps, $\{ cK_{n} \}_{n \geq 1}$. The central contraction map $cK_{n}$ is a map of the following form
\[ cK_{n}: \mathbb{Z} \left< \{ 0,1 \}^{2n} \right> \longrightarrow \mathbb{Z} \left< \{ 0,1 \}^{2n-1} \right>,   \]
\begin{align*}
    cK_{n} \bigl( (\epsilon_1, \ldots, 0, 1, \ldots, \epsilon_{2n}) \bigl) &= (\epsilon_{1}, \ldots, \epsilon_{n-1}, 1 , \epsilon_{n+2}, \ldots, \epsilon_{2n}), \\ 
    cK_{n} \bigl( (\epsilon_1, \ldots, 1, 0, \ldots, \epsilon_{2n}) \bigl) &= (\epsilon_{1}, \ldots, \epsilon_{n-1}, 1 , \epsilon_{n+2}, \ldots, \epsilon_{2n}), \\ 
    cK_{n} \bigl( (\epsilon_1, \ldots, 1, 1, \ldots, \epsilon_{2n}) \bigl) &= (\epsilon_{1}, \ldots, \epsilon_{n-1}, 0 , \epsilon_{n+2}, \ldots, \epsilon_{2n}), \\ 
    cK_{n} \bigl( (\epsilon_1, \ldots, 0, 0, \ldots, \epsilon_{2n}) \bigl) &= 0. 
\end{align*}

\subsection*{The inductive definition of $C^\prime_{n+2}$.} Now we are in the position to write down the formula of $C^\prime_{n+2}$ in terms of $C^\prime_{n}$:
\[C^\prime_{n+2}=   C^\prime_{n} \circ \left( R_{n} \otimes id \otimes R_{n} \right)  \circ cK_{n+1} \circ \left( E_{n+1} \otimes E_{n+1} \right) \circ mK_{n+1}   \circ  K_{2n+5}, \]
where $R_{n}$, $K_{n}$, $E_{n}$ and $E_{n}^{i}$ have been defined in the previous section. 

\section{Proofs of Theorems \ref{thm2}, \ref{thm3} and \ref{thm5}}\label{sec:5}

In this section we prove Theorems \ref{thm2}, \ref{thm3} and \ref{thm5}. We shall rely on the machinery that we developed in a previous paper \cite{PaulSaikia}, which we describe in Subsection \ref{sec:novel}. But first, we recast the problem of enumerating domino tilings into a perfect matching problem. A perfect matching of a graph is a matching in which every vertex of the graph is incident to exactly one edge of the matching. It is easy to see that domino tilings of a region can be identified with perfect matchings of its planar dual graph, the graph that is obtained if we identify each unit square with a vertex and unit squares sharing an edge is identified with an edge. See the bottom left of Figure \ref{fig:ad} for the planar dual graph of an Aztec Diamond of order $4$, and the bottom right of Figure \ref{fig:ad} for the perfect matching corresponding to the tiling shown in the top right of Figure \ref{fig:ad}. So, enumerating domino tilings of Aztec Diamonds is equivalent to calculating perfect matchings of the relevant graphs. With this in mind, we proceed with giving an algebraic structure to our problem.

\subsection{Matching Algebras and Other Results}\label{sec:novel}

Assume $G_1=(V_1,E_1)$ and $G_2=(V_2, E_2)$ are two graphs with a choice of $n$-many distinguished vertices of $V_i$ for $i=1, 2$. Call them $\{x_1, \ldots, x_n \} \subset V_1$ and $\{y_1, \ldots, y_n \} \subset V_2$ respectively. We define the \textit{connected sum} along these distinguished vertices as   \[ G_1 \# G_2 := G_1 \cupdot G_2 / x_i \sim y_i , \forall 1 \leq i \leq n \] 
where $\cupdot$ denotes the disjoint union.  
The vertex set of $G_1 \# G_2$ is defined as $V_1 \cupdot V_2 / x_i \sim y_i$, for all $1 \leq i \leq n $ and the edge set as $E_1 \cupdot E_2 $. 
For any graph $G$, let $M(G)$ denote the number of perfect matchings of $G$. The goal is to understand $M(G_1 \# G_2)$ in terms of $M(G_1)$ and $M(G_2)$.

We define an algebra structure that captures the behavior of  $M(G_1 \# G_2)$ along the boundary. 
Define the \textit{matching algebra} $\mathcal{M}$ over $\mathbb{Z}$ with two generators $y$ and $n$ given by the following relations
\[ \mathcal{M} := \mathbb{Z} \left<y,n \right>/ \left<   yn=yn=n; n^2=0; y^2=y \right>\]
where $G(\epsilon_1, \ldots , \epsilon_{n}) \geq 0$. Here the variable $y$ indicates \textit{yes} or \textit{presence} and the variable $n$ indicates \textit{no} or \textit{absence}. For a graph $G$ with a choice of $n$ distinguished vertices $\{x_1, \ldots, x_n \} \subset V $ we shall assign an element $v_{G} \in \mathcal{M}^{\otimes n} $ of the form \[ v_{G}= \sum_{\epsilon_{i} \in \{ y,n \}} G(\epsilon_1, \ldots ,\epsilon_n) \epsilon_1 \otimes \cdots \otimes \epsilon_n. \] Define an involution $ \bar{\epsilon}$ which switches $y$ to $n$ and $n$ to $y$. That is, as a set we have \[ \{\epsilon, \bar{\epsilon} \}= \{ y,n\}. \] The main result in our previous work \cite{PaulSaikia} was the following theorem.
\begin{theorem}[Theorem 1, \cite{PaulSaikia}]\label{thm:main}
    For graphs $G_{1}= (V_1 , E_1)$ and $G_2=(V_2, E_2)$ with a choice of distinguished vertices as mentioned above, let $v_{G_i}$ denote the element defined as in the previous paragraph. Then  
    \[ M(G_1 \# G_2)= \sum G_{1}(\epsilon_1, \ldots, \epsilon_n ) G_{2}(\bar{\epsilon_1}, \ldots,  \bar{\epsilon_n}). \]
\end{theorem}

For any graph $G= (V,E)$ with $n$ distinguished vertices $\{v_1, \ldots, v_n \} \subset V$ we associate a non-negative integer $G(\epsilon_1, \ldots, \epsilon_n)$ defined as the number of perfect matchings of the subgraph of $G$ where we delete the $i$-th vertex if $\epsilon_i = y$ or we keep it if $\epsilon_i= n$. Note that we have the following identity
 \[G(n,\ldots, n)= M(G). \]
 Said differently, $\epsilon_i=y$ indicates the situation where the identified vertex $v_i$ is available to match with a vertex in the other graph in the connected sum. Whereas, $\epsilon_i=n$ indicates the situation where the identified vertex $v_i$ is already matched with some other vertex in $G$.  With this understanding, we associate the element $v_{G} \in \mathcal{M}^{\otimes n}$ which is a sum of $2^{n}$ terms. We call this the \textit{state sum expansion} of $G$ associated to $\{v_1, \ldots, v_n \}$.

Let $(G_1; v_1, \ldots, v_{i+j})$ and $(G_2; w_1,  \ldots, w_{j+k})$ be two graphs with $ \left(i+j \right)$ and $ \left( j+k \right)$ distinguished vertices respectively. Assume $v_{G_1}$ and $v_{G_2}$ denote the state sum decomposition of $G_1$ and $G_2$ respectively. We construct a new graph $G$ with $ \left(i+j+k \right)$ distinguished vertices defined as follows: 
\begin{equation}\label{eq:cc}
G := G_1 \cupdot G_2 / v_{i+1} \sim w_1, \cdots , v_{i+j} \sim w_j.
\end{equation}

We want to express $v_{G}$ in terms of $v_{G_{1}}$ and $v_{G_{2}}$. To understand it concretely, we need the notion of internal multiplication. Let $I \subset [n]$ and $J \subset [m]$ with $|I|=|J|$ where we denote the set $\{1,2,\ldots, k\}$ by $[k]$. We can define an \textit{internal multiplication} $\phi_{I,J}$ as follows: 
\[ \phi_{I,J}: \mathcal{M}^{\otimes n} \otimes \mathcal{M}^{\otimes m} \to \mathcal{M}^{\otimes (n+m- |I|)},\] 
where the internal multiplication occurs only on the coordinates of $I$ and $J$. More explicitly, if \[|I|=|J|=k, \quad I = \{ i_1 < \cdots < i_k \}, \quad \text{and} \quad J= \{ j_1 < \cdots < j_k \}.\] Take two simple tensors $ \left( A \cdot \epsilon_{1} \otimes \cdots \otimes \epsilon_{n} \right) \in \mathcal{M}^{\otimes n}$ and $\left( B \cdot \eta_{1} \otimes \cdots \otimes \eta_{m}\right) \in \mathcal{M}^{\otimes m}$, where $\epsilon_i, \eta_{j} \in \{y, n\}$. Then,
\begin{multline*}
    \phi_{I,J} \left( (A \cdot \epsilon_1 \otimes \cdots \otimes \epsilon_n) \otimes (B \cdot \eta_1 \otimes \cdots \otimes \eta_{m}) \right)\\:= \left( A \cdot B \right) \epsilon_1 \otimes \cdots \otimes \epsilon_{i_1-1}\otimes (\epsilon_{i_1} \cdot \eta_{j_1}) \otimes \cdots \otimes (\epsilon_{i_{k}} \cdot \eta_{j_{k}}) \otimes \cdots \otimes \epsilon_{n} \otimes \eta_1 \otimes \cdots \otimes \widehat{\eta_{j_1}} \otimes \cdots \otimes \widehat{\eta_{j_k}} \otimes \cdots \otimes \eta_{m},
\end{multline*}

\noindent where~~$\widehat{\cdot}$~~denotes the absence of the variable. For general elements $W \in \mathcal{M}^{\otimes n}$ and $V \in \mathcal{M}^{\otimes m}$ , we extend the definition of internal multiplication 
\(\phi_{I,J} ( W \otimes W ) \) by distributivity. We give an example of the internal multiplication:  for $n=3 , m=4$ and  for $I= \{ 2,3 \},  J=\{ 1,4 \}$:
\[ \phi_{I,J} \left( ( \epsilon_1 \otimes \epsilon_2 \otimes \epsilon_3) \otimes (\eta_1 \otimes \eta_2 \otimes \eta_3 \otimes \eta_4) \right) = \epsilon_1 \otimes (\epsilon_2 \cdot \eta_1) \otimes (\epsilon_3 \cdot \eta_4) \otimes \eta_2 \otimes \eta_3. \]
If $I,J$ are well understood we shall remove them from $\phi$.

Next we state a fundamental lemma which will be useful in the next subsection.
\begin{lemma}[The Patching Lemma, Lemma 3 \cite{PaulSaikia}]\label{lem3}
For $I=\{i+1, \ldots, i+j \} \subset [i+j]$ and $J= \{1, \ldots, j \} \subset [j+k]$, the state sum decomposition of the graph $G$, where $G= (V,E)$ with $n$ distinguished vertices $\{v_1, \ldots, v_n \} \subset V$ is given by  \[v_{G}= \phi_{I,J} (v_{G_1} \otimes v_{G_2}),\] where $G_1$ and $G_2$ are as described in \eqref{eq:cc}.  
\end{lemma}

\begin{figure}[htb!]\label{fig1_factor}
\includegraphics[scale=1]{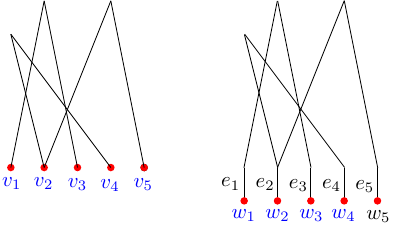}
\caption{Example of $1$-factor addition.}
\end{figure}

\begin{remark}
    In our application, if the distinguished vertices are well understood we use the dot symbol `$\cdot$' to indicate the interior multiplication.
\end{remark}

Given a graph $G=(V,E)$ with a choice of distinguished vertices $\{ v_1, \ldots, v_n \} \subset V $ we can construct a new graph called \textit{$1$-factor addition of $G$} where we add a separate pendent edge $e_i$ for each distinguished vertex $v_i$. We label the degree $1$ vertex corresponding to $v_i$ as $w_i$. For example, see Figure \ref{fig1_factor} where the distinguished vertices are colored red. If  \[G^\prime= \left( V \cup \{w_1, \ldots, w_n \}, E \cup \{e_1, \ldots, e_n \} \right)\]denotes the $1$-factor addition of $G$ with the choice of distinguished vertices as $\{w_1, \ldots, w_n \}$ then we want to understand how the state sum decomposition  $v_{G}$ and $v_{G^\prime}$ related. The answer is provided in the next lemma.
%%%%%%%%%%%%%%%%%%%%%%%%%%%%%%%%%%%%%%%%%%%%%%%%%%%%%%%%%%%%%%%%%%%%%
%%%%%%%%%%%%%%%%Transfer Lemma %%%%%%%%%%%%%%%%%%%%%%%%%%%%%%%%%%%%%%%
%%%%%%%%%%%%%%%%%%%%%%%%%%%%%%%%%%%%%%%%%%%%%%%%%%%%%%%%%%%%%%%%%%%%%%%%
\begin{lemma}[Transfer lemma]\label{lem4}
If $v_{G}= \sum\limits_{\epsilon_i \in \{ y,n \}} G(\epsilon_1, \ldots, \epsilon_n) \epsilon_1 \otimes \epsilon_{n}$ and $v_{G^\prime}= \sum\limits_{\epsilon_i \in \{ y,n \}} G^\prime(\epsilon_1, \ldots, \epsilon_n) \epsilon_1 \otimes \epsilon_{n}$ then the following holds \[ G(\epsilon_1, \ldots , \epsilon_n)= G^\prime(\bar{\epsilon_1}, \ldots, \bar{\epsilon_n}). \]
\end{lemma}
\begin{proof}
    The set of distinguished vertices $\{ v_1, \ldots, v_{n} \}$ in $G$ no longer remains distinguished vertices in the $1$-factor addition $G^{'}$. Adding an edge $e_{i}$ corresponds to taking a connected sum of $G$ with $L_{2}$ where the vertex $v_{i}$ is identified with one of the boundary vertex of $L_{2}$. The state sum decomposition of $L_2$ is given by 
    \[ v_{L_2}= y \otimes y + n \otimes n\]
    Finally,
    \begin{align*}
        y \cdot(y\otimes y + n \otimes n) &= y^{2} \otimes y + (y\cdot n) \otimes n = y \otimes y + n \otimes n, \\ 
        n \cdot(y\otimes y + n \otimes n) &= (n \cdot y) \otimes y + n^{2} \otimes n =  n \otimes y.
    \end{align*}
    Thus, for simple tensor $\epsilon_1 \otimes \cdots \otimes \epsilon_{n}$ appearing in the state sum decomposition of $G$, if $\epsilon_{i}=y$  then it transfers to $n$ in $G^{'}$. Similarly, $\epsilon_{i}=n$ transfers to $y$ in $G^{'}$. This is because, in the state sum decomposition of a graph the non-distinguished vertices must be already matched. Finally, by Lemma \ref{lem3}, we conclude our proof. 
\end{proof}
We define an element $\mathcal{F}_{n} \in \mathcal{M}^{\otimes n}$ as given by the stipulation that $n$ must occur in pairs in its simple tensor decomposition. Here are some examples
\begin{align*}
    \mathcal{F}_{1} &= y + n, \\
    \mathcal{F}_{2} &= y \otimes y + n \otimes n ,\\ 
    \mathcal{F}_{3} &= y \otimes y \otimes y + y \otimes n \otimes n + n \otimes n \otimes y, \\
    \mathcal{F}_{4} &= y \otimes y \otimes y \otimes y + n \otimes n \otimes y \otimes y + y \otimes n \otimes n \otimes y + y \otimes y \otimes n \otimes n + n \otimes n \otimes n \otimes n. 
\end{align*}
Note that the number of simple tensors in $\mathcal{F}_{n}$ (that is, the number of terms in $\mathcal{F}_n$) is exactly equal to the Fibonacci number $F_{n}$, where we define the sequence as
\[
F_1=1, \quad F_2=2, \quad F_n=F_{n-1}+F_{n-2} \quad \text{for all $n\geq 3$}.
\]
\begin{lemma}
    The state sum decomposition of the line graph $L_{n}$ with $n$ vertices where all of the vertices are also distinguished vertices is given by $\mathcal{F}_{n}$.
\end{lemma}
\begin{proof}
  We use induction on $n$ and Lemma \ref{lem3}. Note that the definition of $\mathcal{F}_{n}$ implies the following for all $n\geq 2$\[
 \mathcal{F}_{n}= \mathcal{F}_{n-1} \otimes y + \mathcal{F}_{n-2} \otimes n \otimes n. 
  \]
  Now by Lemma \ref{lem3}, the state sum decomposition of the line graph $L_{n+1}$ is given by the internal multiplication of $\mathcal{F}_{n}$ with $\mathcal{F}_{1}$ where we need to multiply the last coordinate of $\mathcal{F}_{n}$ with the first coordinate of $\mathcal{F}_{1}$ via the matching algebra $\mathcal{M}$. Thus, we have
  \begin{align*}
      \mathcal{F}_{n}\cdot \mathcal{F}_{1}&=  \left( \mathcal{F}_{n-1} \otimes y + \mathcal{F}_{n-2} \otimes n \otimes n \right) \cdot ( y \otimes y + n \otimes n) \\
      &= \mathcal{F}_{n-1} \otimes (y^2) \otimes y + \mathcal{F}_{n-1} \otimes (y \cdot n) \otimes n + \mathcal{F}_{n-2} \otimes n \otimes (n \cdot y) \otimes y + \mathcal{F}_{n-2} \otimes n \otimes (n \cdot n) \otimes n\\
      &= \mathcal{F}_{n-1} \otimes y \otimes y + \mathcal{F}_{n-1} \otimes n \otimes n + \mathcal{F}_{n-2} \otimes n \otimes n \otimes y \\
      &= (\mathcal{F}_{n-1} \otimes y + \mathcal{F}_{n-2} \otimes n \otimes n) \otimes y + \mathcal{F}_{n-1} \otimes n \otimes n \\
      &= \mathcal{F}_{n} \otimes y + \mathcal{F}_{n-1} \otimes n \otimes n \\
      &= \mathcal{F}_{n+1}.
  \end{align*}
  This completes the proof.
\end{proof}

Armed with this lemma, we can finally prove our results. 

\subsection{Proof of Theorems \ref{thm2} and \ref{thm3}}

Observe that finding the number of diagonally symmetric, that is $\left< t \right>$- invariant domino tilings of $\ad(n)$ amounts to computing the state sum decomposition of the class of graphs shown in Figure \ref{fig:1}. We call this collection of graphs $ \{ \mc_{n} \}_{n \geq 1}$, for $n=1, 2, 3$ and $4$ the graphs are as shown in Figure \ref{fig:1}. Here and elsewhere we only mark the distinguished vertices (in red) that we will use; all other vertices in the graphs are unmarked. We prove the following theorem.

\begin{figure}[htb!]\label{fig:1}
\includegraphics[scale=1]{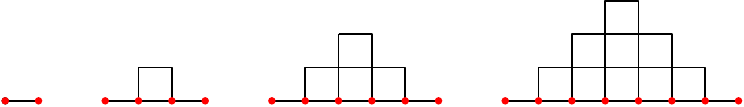}
\caption{Dual Graph associated with Diagonally Symmetric Domino Tilings of Aztec Diamonds.}
\end{figure}

\begin{theorem}\label{thm6}
    The sum of the coefficients in the state sum decomposition of $\mc_{n}$ enumerates the $\left<  t \right>$- invariant domino tilings of $\ad(n)$. The sum of the squares of the coefficients in the state sum decomposition of $\mc_{n}$ enumerates the total number of domino tilings of $\ad(n)$. 
\end{theorem}
\begin{proof}
  Taking the graphical dual of $\ad(n)$ we note that the number of $\left<  t \right>$- invariant domino tilings of $\ad(n)$ is the same as the number of perfect matchings symmetric with respect to the central line. For $n=4$, refer to Figure \ref{fig:2}.
    
    \begin{figure}[htb!]\label{fig:2}
        \includegraphics[scale=1]{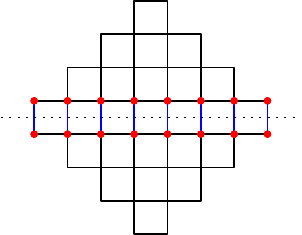}
        \caption{Dual graph of $\ad(n)$ and its relation to $\langle t \rangle$-invariant tilings.}
    \end{figure}
    
    Note that any matching of the top $\mc_{n}$ graph corresponding to the defect $ \epsilon_1 \otimes \cdots \otimes \epsilon_{2n}$ extends uniquely to the bottom $\mc_{n}$ with the same defect $ \epsilon_1 \otimes \cdots \otimes \epsilon_{2n}$. For the centrally symmetric perfect matchings of $\ad(n)$ the bottom $\mc_{n}$ is completely determined by the choice of top $\mc_{n}$.  Thus,  we have the equality.
    
    For the second part, this follows from Lemma \ref{lem4} and Theorem \ref{thm:main} above. 
    %%%%%%%%%Put the reference%%%%%%%%%%%%%%%%%%
    %%%%%%%%%%%%%%%%%%%%%%%%%%%%%%%%%%%%%%%%%%%%
\end{proof}

Next, we relate the state sum decomposition of $\mc_{n}$ with the function $C_n$ discussed in Section \ref{sec:two}. We make the following identifications 
\begin{align*}
    y & \longleftrightarrow \mathcal{V}(y)= 1, \\
    n & \longleftrightarrow \mathcal{V}(n)= 0, \\
    \left( \epsilon_1 \otimes \cdots \otimes \epsilon_n \right) & \longleftrightarrow \left( \mathcal{V}(\epsilon_1), \ldots, \mathcal{V}(\epsilon_n) \right).
\end{align*}
Here, $\mathcal{V}$ stands for \textit{valuation}. The valuation map $\mathcal{V}$, allows us to transforms the simple tensor $\epsilon_1 \otimes \epsilon_n$ where $\epsilon_i \in \{ y, n \}$ to a sequence in $\{ 0, 1 \}^{n}$. The relationship between $\mc_{n}$ and $C_{n}$ is explicitly described by the following lemma. 

%%%%%%%%%%%%%%%%%%%%%%%%%%%%%%%%%%%%%%%%%%%%%%%%%%%%%%%%
%%%%%%%%%%%%%%%%%%%%%%%%%%%%%%%%%%%%%%%%%%%%%%%%%%%%%%%%
%%%%%%%%%%Lemma 5%%%%%%%%%%%%%%%%%%%%%%%%%%%%%%%%%%%%%%%

\begin{lemma}\label{lem5}
    The coefficient of $\epsilon_1 \otimes \cdots  \otimes \epsilon_{2n}$ in the state sum decomposition of $\mc_n$ is exactly equal to $C_{n} \bigl( (\mathcal{V}(\epsilon_1), \ldots, \mathcal{V}(\epsilon_{2n}) ) \bigl)$.
\end{lemma}
\begin{proof}
    We use induction on $n$. For $n=1$ and $n=2$ this relationship can easily be verified. We assume the result is true for $n-1$. We now build $\mc_{n}$ from $\mc_{n-1}$ through the following steps as illustrated for $n-1=3$ in Figures \ref{fig:setp-1} -- \ref{fig:setp-4}.
    \begin{itemize}
        \item[Step 1] We start with the graph $\mc_{n-1}$.
        \item[Step 2] We append pendant edges in $\mc_{n-1}$ along the distinguished vertices (see Figure \ref{fig:setp-2}).
        \item[Step 3] We take the connected sum of $\mc_{n-1}$ and $L_{2n-2}$ along the new distinguished vertices (see Figure \ref{fig:setp-3}).
        \item[Step 4] We take the connected sum of $L_1$ from the right and from the left of the graph obtained in the previous step (see Figure \ref{fig:setp-4}).
    \end{itemize}
    
    \begin{figure}[H]
        \includegraphics[]{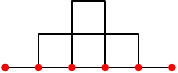} 
        \caption{Step 1: We start with $\mc_{n-1}$ (here $n=4$).}
        \label{fig:setp-1}
    \end{figure}

       \begin{figure}[H]
       \centering
       \includegraphics[]{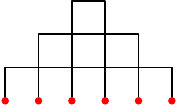}
    \caption{Step 2: We append pendant edges in $\mc_{n-1}$ (here $n=4$).}
    \label{fig:setp-2}
   \end{figure}

   \begin{figure}[H]
       \centering
       \includegraphics[]{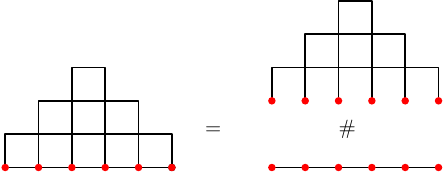}
      \caption{Step 3: We take the connected sum of $1-factor \, \, addition \, \, of \mc_{n-1}$ and $L_{2n-2}$ (here $n=4$).}
      \label{fig:setp-3}
   \end{figure}

     \begin{figure}[H]
       \centering
       \includegraphics[]{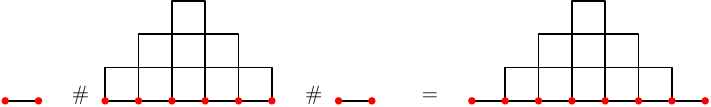}
      \caption{Step 4: We take the connected sum of $L_1$ from the right and from the left of the graph obtained in Step 3.}
      \label{fig:setp-4}
   \end{figure}

Note that the transition of the state sum decomposition from Step 1 to Step 2 is given by Lemma \ref{lem4}. The reverse map $R_{n}$ defined in section $2$ precisely captures the change.

Using Lemma \ref{lem4}, we conclude that the transition of the state sum decomposition from Step 2 to Step 3 is given by the interior multiplication by the Fibonacci element $\mathcal{F}_{2n-2}$ defined earlier. In order to determine the coefficient of $\epsilon_1 \otimes \cdots \otimes \epsilon_{2n-2}$, we need to understand all possible ways to  construct $\epsilon_1 \otimes \cdots \otimes \epsilon_{2n-2}$ through the interior multiplication by $\mathcal{F}_{2n-2}$. By Lemma \ref{lem4} we have the following identity
\[ \mathcal{F}_{2n-2}= \underbrace{\mathcal{F}_{2} \cdots \mathcal{F}_{2}}_{(2n-3) \, \,  \text{times}} = \left( \mathcal{F}_{2} \right)^{2n-3},  \]
where the interior multiplication of consecutive pairs of $\mathcal{F}_{2}$ multiplies the second factor of the first $\mathcal{F}_{2}$ with the first factor of second $\mathcal{F}_{2}$, that is
\[ \mathcal{F}_{2} \cdot \mathcal{F}_{2} = (y \otimes y + n \otimes n) \cdot (y \otimes y + n \otimes n)= y \otimes (y \cdot y) \otimes y + y \otimes (y \cdot n) \otimes n + n \otimes (n \cdot y) \otimes y + n \otimes (n \cdot n) \otimes n. \]

Next, we analyse how multiplication by $\mathcal{F}_{2}$ affect simple tensors on both sides: 
\begin{align*}
    n \cdot (y \otimes y + n \otimes n) \cdot n &= n \otimes n, \\ 
    n \cdot (y \otimes y + n \otimes n) \cdot y &= n \otimes y, \\ 
   y \cdot (y \otimes y + n \otimes n) \cdot n &= y \otimes n, \\
   y \cdot (y \otimes y + n \otimes n) \cdot y &= y \otimes y + n \otimes n.
\end{align*}
Note that $n \otimes n $ can appear in two ways, either from $n \cdot \mathcal{F}_{2}  \cdot n$ or from $y \cdot \mathcal{F}_{2} \cdot y$. Thus in the opposite direction, in terms of sequences $(n \otimes n) \longleftrightarrow (0,0) $ expands to $(0,0)+ (1,1)$. Similarly $(1,0), (0,1)$ and $(1,1)$ uniquely expands to themselves.  Using the same reasoning for each of the $2n-3$ middle $\mathcal{F}_{2}'s$, the expansion map $E_{2n-3}$ captures the phenomenon where the basic expansion $E^{i}_{2n-3}$ captures the $i$-th $\mathcal{F}_{2}$.

The reasoning in the last step is similar. It follows from the computation below:
\begin{align*}
    y \cdot (y \otimes y + n \otimes n)= y \otimes y + n \otimes n \, \,   &\text{AND}  \, \, n \cdot (y \otimes y + n \otimes n) = n \otimes y, \\ 
     (y \otimes y + n \otimes n) \cdot y = y \otimes y + n \otimes n \, \,   &\text{AND}  \, \,  (y \otimes y + n \otimes n) \cdot n = y \otimes n. 
\end{align*}
In the opposite direction both $y \otimes y$ and $n \otimes n $ contracts to $y$ on either end. Similarly, if the simple tensor starts with $y \otimes n $ or ends with $n \otimes y$ then it must contract to $n$ on either end. These observations are precisely encapsulated by the map $K_{2n}$ after the translation into $\{ 0,1 \}$ sequences. 

Going in the opposite direction we conclude the values of $C_{n}=   C_{n-1} \circ R_{2n-2} \circ E_{2n-2} \circ K_{2n}$ exactly captures the coefficients in the state sum decomposition of $\mc_{n}$. 
\end{proof}

\begin{proof}[Proof of Theorems \ref{thm2} and \ref{thm3}]
Combining Theorem \ref{thm6} and Lemma \ref{lem5} we conclude the proofs of Theorems \ref{thm2} and \ref{thm3}.    
\end{proof}

\begin{remark}
   We note that, the largest numbers appearing in the state sum decomposition of $\mc_n$ are the Large Schr\"oder numbers. This is due to the fact that these numbers count the number of perfect matchings in the cellular graph formed by a triangular grid of $n$ squares (see, the work of Ciucu \cite{zcell} for more details).
\end{remark}

We now prove Lemma \ref{lem1}.
\begin{proof}[Proof of Lemma \ref{lem1}]
Note that for any sequence 
$\left( \epsilon_0, \ldots ,\epsilon_{n-1} \right) \in \{0,1 \}^{n}$,  we can associate a unique integer by treating it as a binary expansion, that is 
\[ \left( \epsilon_0, \cdots ,\epsilon_{n-1} \right) \longleftrightarrow \sum_{i=0}^{n-1} \epsilon_{i} \cdot 2^{i}.\]
We prove the following observation: For $(\epsilon_{1}, \ldots, \epsilon_{2n}) \in \{0,1 \}^{2n}$ such that  $(\epsilon_{1}, \ldots, \epsilon_{2n}) \in \supp(C_{n}) $, we have $ 3| \sum_{i=1}^{2n} \epsilon_{i} \cdot 2^{i-1} $.

Lemma \ref{lem1} follows quite easily from this observation. We prove this observation using induction on $n$. For $n=1$, it can be seen quite easily. Our strategy for the general step would be to follow the steps described in Lemma \ref{lem5}. We now show that the divisibility condition remains the same in each intermediate step. 

Assume the condition remains true for $C_{n}$. The divisibility condition in Step 1 is the induction hypothesis. In Step 2, the element $ \left( \epsilon_1, \ldots, \epsilon_{2n} \right)$ transforms to $ \left( 1-\epsilon_1, \ldots, 1-\epsilon_{2n} \right)$. The divisibility condition follows from the following fact : \[ 1+ 2+ 2^{2}+ \cdots + 2^{2n-1}= 2^{2n}-1 = 4^{n}-1  \equiv 0 \pmod 3. \]

In Step 3, when we have the consecutive terms in the sequence $\epsilon_i=0, \epsilon_{i+1}=1$ or $\epsilon_{i}=1, \epsilon_{i+1}= 0$ or $\epsilon_{i}=\epsilon_{i+1}=0$, it remains same after the interior multiplication by $\mathcal{F}_{2n}$. However if the consecutive terms are $ \epsilon_{i}= \epsilon_{i+1}=1$ then it branches off to two different elements with corresponding terms being $\epsilon_{i}=\epsilon_{i+1}=1$ and $\epsilon_i= \epsilon_{i+1}=0$. In the first case, the divisibility condition is trivial. In the second case, the associated numbers differ by $2^{i+1}+2^{i}=3\cdot 2^{i}$, which is  divisible by $3$.

Finally, we verify the divisibility condition in Step 4. For the sequence $\left( \epsilon_1, \epsilon_2, \ldots, \epsilon_{2n-1}, \epsilon_{2n} \right) \in \{0,1 \}^{2n} $ we denote the associated number by $A= \sum_{i=1}^{2n} \epsilon_i \cdot 2^{i-1}$. Then we have the following
\begin{itemize}
    \item  If $\epsilon_1=\epsilon_{2n}=0$, then the sequence transforms to $(1, \epsilon_{1}, \epsilon_2, \ldots, \epsilon_{2n-1}, \epsilon_{2n}, 1)$. The associated number for the new sequence is given by $2\cdot A + 1 + 2^{2n+1} \equiv 0  \pmod 3$ when $A \equiv 0 \pmod 3$.
    \item If $\epsilon_1= 1$ and $\epsilon_{2n}=0$, then the sequence transforms to two new sequences: \[(1, \epsilon_1, \epsilon_2, \ldots, \epsilon_{2n-1},\epsilon_{2n}, 1)\quad \text{and}\quad(0, 1- \epsilon_1, \epsilon_2, \ldots, \epsilon_{2n-1},\epsilon_{2n}, 1).\] In the first one,  the associated number is given by $2\cdot A + 1+ 2^{2n+1}$. In the second one, the associated number is given by $2 \cdot A -2 + 2^{2n+1} \equiv 0 \pmod 3$ when $A \equiv 0 \pmod 3$. 
    \item The case $\epsilon_1=0$ and $\epsilon_{2n}=1$ is similar to the previous case. 
    \item Finally with $\epsilon_1= \epsilon_{2n}=1$ the sequence transforms to four different sequences: 
    \begin{align*}
            \left( 1, \epsilon_1, \epsilon_2, \cdots, \epsilon_{2n-1}, \epsilon_{2n}, 1 \right)&, \left( 0, 1- \epsilon_1, \epsilon_2, \cdots, \epsilon_{2n-1}, \epsilon_{2n}, 1 \right) \\ \left( 1, \epsilon_1, \epsilon_2, \cdots, \epsilon_{2n-1}, 1-\epsilon_{2n}, 0 \right)&,  \left( 0, 1-\epsilon_1, \epsilon_2, \cdots, \epsilon_{2n-1}, 1-\epsilon_{2n}, 0 \right). 
    \end{align*}
    The associated numbers are given by $2\cdot A + 1+ 2^{2n+1}$, $2 \cdot A -2 + 2^{2n+1}$, $2 \cdot A +1 - 2^{2n}$ and $2A - 2- 2^{2n}$ respectively. The divisibility condition can be checked similarly. \qedhere
\end{itemize}
\end{proof}

\subsection{Proof of Theorem \ref{thm5}}

Observe that finding the number of diagonally and anti-diagonally symmetric, that is $\left<r^2, t \right>$- invariant domino tilings of $\ad(n)$ amounts to computing the state sum decomposition of the class of graphs shown in Figure \ref{fig:omc}. We call this collection of graphs $\{ \overline{\mc}_{n} \}_{n \geq 1}$, for $n=1, 2, 3$ and $4$ the graphs are as shown in Figure \ref{fig:omc}. 

\begin{figure}[htb!]\label{fig:omc}
\includegraphics[scale=1]{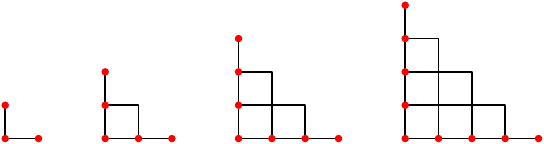}
\caption{Dual Graph associated with Diagonally \& Anti-diagonally Symmetric Domino Tilings of Aztec Diamonds.}
\end{figure}

Like earlier, we have the following result.
\begin{theorem}\label{thm7}
    The sum of the coefficients in the state sum decomposition of $\overline{\mc}_{n}$, weighted according to whether the middle distinguished vertex (that is, the one on the bottom left corner) corresponds to a $n$ or a $y$ with weights $1$ and $2$ respectively, enumerates the $\left< r^2, t \right>$- invariant domino tilings of $\ad(n)$.
\end{theorem}

\begin{proof}
Dualizing the picture of $\ad(n)$ we note that the number of $\left< r^2, t \right>$- invariant domino tilings of $\ad(n)$ is the same as the number of perfect matchings symmetric with respect to the dotted lines shown in Figure \ref{fig:dot} (here we take $n=5$, and ignore the red dotted square for the moment).   
     \begin{figure}[htbp!]\label{fig:dot}
       \centering
       \includegraphics[scale=1]{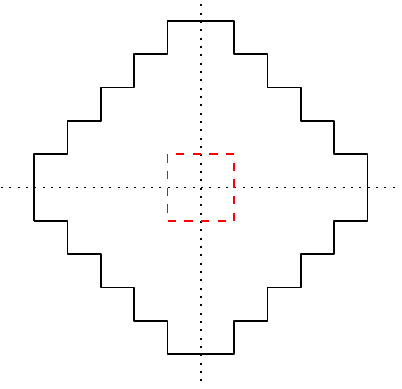}
      \caption{Aztec Diamond with diagonal and anti-diagonal symmetry highlighted.}
      \label{fig:dad}
   \end{figure}

   The red dotted square in the middle is special. Any $\langle r^2, t\rangle$-invariant domino tiling will either contain this square as a sub-tiling or it will not contain the square as a sub-tiling (that is, the unit squares making up this $2\times 2$ square will tile with unit squares which lie outside the $2\times 2$ square). The first case corresponds to a $y$ assigned to the bottom left corner distinguished vertex of $\omc_n$, while the second case corresponds to a $n$ being assigned to this vertex. In the first case, the $2$ tilings of the $2\times 2$ square correspond to a weight of $2$ for the relevant coefficient in the state sum decomposition of $\omc_n$, while the other case, corresponds to to a weight of $1$ for the relevant coefficient in the state sum decomposition of $\omc_n$. This completes the proof.
\end{proof}

Next, we relate the state sum decomposition of $\omc_{n}$ with the function $C^\prime_n$ discussed in Section \ref{sec:two}. We again use the \textit{valuation map} defined in the previous subsection. The relationship between $\omc_{n}$ and $C^\prime_{n}$ is explicitly described by the following lemma. 

%%%%%%%%%%%%%%%%%%%%%%%%%%%%%%%%%%%%%%%%%%%%%%%%%%%%%%%%
%%%%%%%%%%%%%%%%%%%%%%%%%%%%%%%%%%%%%%%%%%%%%%%%%%%%%%%%
%%%%%%%%%%Lemma 5%%%%%%%%%%%%%%%%%%%%%%%%%%%%%%%%%%%%%%%

\begin{lemma}\label{lem6}
    The coefficient of $\epsilon_1 \otimes \cdots  \otimes \epsilon_{2n+1}$ in the state sum decomposition of $\omc_{n+1}$ is exactly equal to $C^\prime_{n} \bigl( (\mathcal{V}(\epsilon_1), \ldots, \mathcal{V}(\epsilon_{2n+1}) ) \bigl)$.
\end{lemma}

\begin{proof}
   The proof is again by using induction on $n$. For $n=1$ and $n=2$ this relationship can easily be verified. We assume the result is true for $n-1$. We now build $\omc_{n+1}$ from $\omc_{n-1}$ through the following steps as illustrated for $n-1=5$ in Figure \ref{fig:steps-dad}.
    \begin{itemize}
        \item[Step 1] We start with the graph $\omc_{n-1}$ (we mark the vertices $1, 2, \ldots, 2n-3$ as shown in Figure \ref{fig:steps-dad}).
        \item[Step 2] We append pendant edges in $\omc_{n-1}$ along the distinguished vertices, except for the vertex in the bottom-left corner.
        \item[Step 3] We take the connected sum of $\omc_{n-1}$ and $L_{3}$ along the middle vertex of $L_3$ and the the vertex in the bottom-left corner of $\omc_{n-1}$.
        \item[Step 4] We take the connected sum of $L_{n-1}$ from the left and from the bottom of the graph obtained in the previous step (see the middle figure of the second row in Figure \ref{fig:steps-dad}).
        \item[Step 5] We take the connected sum of the graph obtained in the previous step and $L_{3}$ along the extreme vertices of $L_3$ and the vertices in the bottom corners of the graph obtained in the previous step (see left figure of second row in Figure \ref{fig:steps-dad}).
  \item[Step 6] We take the connected sum of $L_2$ and the graph obtained in the previous step, once from the top left corner and once from the bottom right corner (see the left figure of the third row in Figure \ref{fig:steps-dad}). This now results in $\omc_{n+1}$.
    \end{itemize}

        \begin{figure}[htbp!]
       \centering
       \includegraphics[scale=1]{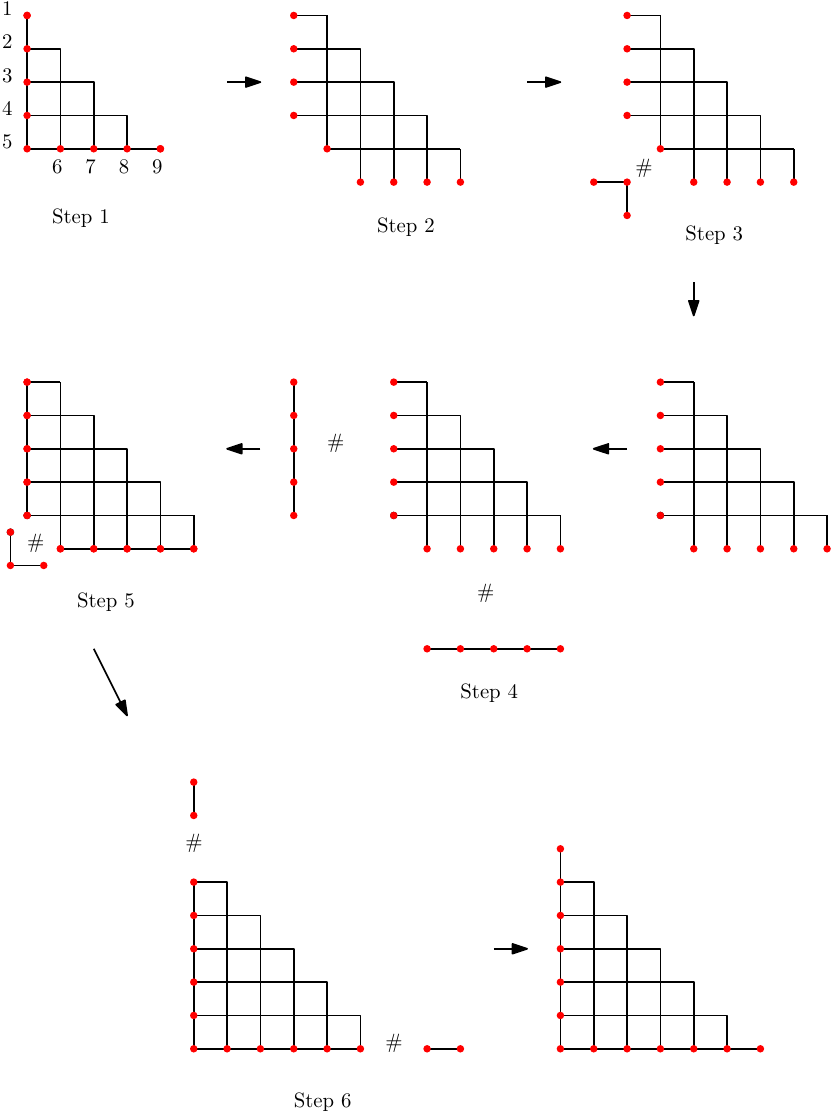}
      \caption{Steps 1 to 6, starting from top left and following the arrows to bottom right.}
      \label{fig:steps-dad}
   \end{figure}

  Since the process was explained in detail in the proof of Lemma \ref{lem5}, we only sketch the relevant details here. The transition from Step 1 to Step 2 is via the map $R_n \otimes \text{id}\otimes R_n$. This is easy to see as we are appending edges only to the top $n$ vertices on the left and right $n$ vertices on the bottom. The transition from Step 2 to Step 3 is described as follows: we are taking the connected sum of $L_{3}$ along the middle vertex. In this process, the distinguished vertex turns to a non-distinguished vertex in the new configuration in step $3$. The boundary vertices of $L_3$ turn into two new additional distinguished vertices. Taking the interior multiplication of the state sum decomposition of $L_3$ along the middle vertex we observe the following multiplication table: 
   \begin{align*}
       y\otimes (y \cdot \mathbf{y}) \otimes y + n \otimes (n \cdot \mathbf{y}) \otimes y + y \otimes (n \cdot \mathbf{y}) \otimes n &= y\otimes y  \otimes y + \mathbf{n \otimes n \otimes y} + \mathbf{ y \otimes n  \otimes n}, \\ 
      y\otimes (y \cdot \mathbf{n}) \otimes y + n \otimes (n \cdot \mathbf{n}) \otimes y + y \otimes (n \cdot \mathbf{n}) \otimes n &= \mathbf{ y\otimes n  \otimes y},
   \end{align*}
   where the $\mathbf{y}$ or $\mathbf{n}$ denotes the state sum decomposition of the middle distinguished vertex in Step 2. Since it turns to a non-distinguished vertex in Step 3, the only relevant terms are those that turn $n$ after the multiplication on the right-hand side. The relevant terms are bold on the right-hand side of the equation. Finally, by the identification \[ y \longleftrightarrow 1 ; n \longleftrightarrow 0,\] the central contraction map $cK_{n}$ exactly captures the situation in the opposite direction. 
  The transition from Step 3 to Step 4 is via the map $E_{n+1}\circ E_{n+1}$ in the opposite direction, this was explained in the proof of Lemma \ref{lem5} going from Step 2 to Step 3 in that case. Similarly, the transition from Step 4 to Step 5 is via the middle contraction map $mK_{n+1}$ in the opposite direction. It follows from observing the multiplication table
  \begin{align*}
      y \cdot \left(y\otimes y  \otimes y + n \otimes n \otimes y +  y \otimes n  \otimes n\right) \cdot y &= \left(y\otimes y  \otimes y + n \otimes n \otimes y +  y \otimes n  \otimes n\right), \\ 
    y \cdot \left(y\otimes y  \otimes y + n \otimes n \otimes y +  y \otimes n  \otimes n\right) \cdot n &= y \otimes y \otimes n + n \otimes n \otimes n, \\ 
    n \cdot \left(y\otimes y  \otimes y + n \otimes n \otimes y +  y \otimes n  \otimes n\right) \cdot y &= n \otimes y \otimes y + n \otimes n \otimes n, \\ 
    n \cdot \left(y\otimes y  \otimes y + n \otimes n \otimes y +  y \otimes n  \otimes n\right) \cdot n &= n \otimes y \otimes n. \end{align*}
  Finally in the final step, Step 6, two pendent edges are being appended. This situation was explained in the proof of Lemma \ref{lem5} going from Step 7 to Step 8 in that case. 
   \end{proof}

\begin{proof}[Proof of Theorem \ref{thm5}]
    Combining Theorem \ref{thm7} and Lemma \ref{lem6} we conclude the proof of Theorem \ref{thm5}.
\end{proof}

\end{document}